\documentclass[sn-mathphys]{sn-jnl}

\jyear{2021}%

\theoremstyle{thmstyleone}%

\theoremstyle{thmstyletwo}%

\theoremstyle{thmstylethree}%

\begin{document}

\title[On the relation of the frame-related operators of fusion frame systems]{On the relation of the frame-related operators of fusion frame systems}
%
%
\author*[1]{\fnm{Lukas} \sur{Köhldorfer}}\email{koehli@gmx.at}
\author[1]{\fnm{Peter} \sur{Balazs}}\email{peter.balazs@oeaw.ac.at}
\affil*[1]{\orgdiv{Acoustics Research Institute}, \orgname{Austrian Academy of Sciences}, \orgaddress{\street{Wohllebengasse 12-14}, \city{Vienna}, \postcode{1040}, \state{Vienna}, \country{Austria}}}
%
%
\abstract{Frames have been investigated frequently over the last few decades due to their valuable properties, which are desirable for various applications as well as interesting for theory. Some applications additionally require distributed processing techniques, which naturally leads to the concept of fusion frames and fusion frame systems.
The latter consists of a system of subspaces, equipped with local frames on each of them, and a global frame. \

In this paper, we investigate the relations of the associated frame-related operators on all those three levels. 
For that we provide a detailed investigation on bounded block diagonal operators between Hilbert direct sums. We give the relation of the frame-related operators of the fusion frame and the corresponding frame systems in terms of operator identities. By applying these identities we prove further properties of fusion frame systems.}
%
%

\maketitle
\section{Introduction}\label{sec1}

Frames are sequences of elements in a (separable) Hilbert space, which provide the possibility to represent and reconstruct vectors in a non-unique, redundant and stable way \cite{duffschaef1,daubgromay86,ole1}. The study of frames led to many interesting theoretical results, such as versions of the \emph{Balian-Low theorem} \cite{HeWe_96,BeHeWal_95}, or \emph{Feichtinger's conjecture} \cite{CasChrLinVer_05,MarSpiSri_15,Gav13}. On the other hand, the properties of frames are desired for a vast number of applications, e.g. signal processing \cite{framepsycho16}, compressed sensing \cite{rascva06} and many more, see \cite{ole1} and the numerous references mentioned there. 

Fusion Frames have been first introduced as \emph{frames of subspaces} in order to give necessary and sufficient conditions for piecing local frames together to a global frame \cite{caskut04}. Since fusion frames generalize the notion of a (classical) frame, many concepts from frame theory can be transferred to the fusion frame setting \cite{caskut04,cakuli08,Asg_15}. However, the aspect of duality for fusion frames is non-trivial as it cannot be directly extended from the classical frame case, which makes the theory much more interesting \cite{kutpatphi17,Gav07,heimoanza14,hemo14,heimo18}. The idea of combining local frames on subspaces to a global frame and performing distributed processing procedures via frames, is modeled by the notion of a fusion frame system, which yields the flexibility to either perform \emph{global reconstruction} or \emph{local reconstruction} of a given signal \cite{cakuli08}. 

On all three layers - subspace, local and global frames - of a fusion frame system, the typical frame-related operators, like the analysis, synthesis and frame operator, can be defined. In this paper we study the relation between these operators associated to a fusion frame system. For that purpose we provide details for bounded block diagonal operators between Hilbert direct sums. In particular, we show that for those operators we get canonical combinations, like combining the fusion synthesis with the block diagonal operator of the local synthesis operators resulting in the global synthesis operator. We show that a global Riesz basis results in a fusion Riesz basis with local Riesz bases, a result where only the opposite direction is known so far. We can extend the answer to the question, when centralized and distributed reconstruction coincide, to the case of a fusion Riesz basis. Finally, we give a collection of results on properties, which are preserved in fusion frame systems. 

\

This paper is an extended version of \cite[Sections 3, 4.3, 4.4]{kohl21}.

\section{Preliminaries}\label{sec2}

Throughout this notes,  $\mathcal{H}$\index{$\mathcal{H}$} is always a separable Hilbert space. If $V$ is a subspace of $\mathcal{H}$ then $\pi _V$\index{$\pi _V$} denotes the orthogonal projection onto $V$. A Hilbert space $\mathcal{H}$ of functions $f:X \longrightarrow \mathbb{F}$ ($\mathbb{F} = \mathbb{R}$ or $\mathbb{F} = \mathbb{C}$) is called a reproducing kernel Hilbert space (RKHS), if all evaluation functionals $\delta_x$ on $\mathcal{H}$ are continuous, i.e. if for every every $x\in X$, there exists $C_x >0$ such that $\vert \delta_x (f) \vert := \vert f(x) \vert \leq C_x \Vert f \Vert$ for all  $f\in \mathcal{H}$ \cite{aro50, paulsen_raghupathi_2016}. The set of positive integers $\lbrace 1, 2, 3, \dots \rbrace$ is denoted by $\mathbb{N}$, and $\delta_{ij}$ denotes the Kronecker-delta. The domain, kernel and range of an operator $T$ is denoted by $\mathrm{dom}(T)$\index{$\mathrm{dom}(T)$}, $\mathcal{N}(T)$\index{$\mathcal{N}(T)$} and $\mathcal{R}(T)$\index{$\mathcal{R}(T)$} respectively.  $\mathcal{I}_X$\index{$\mathcal{I}_X$} denotes the identity operator on a given space $X$. The set of bounded operators between two normed spaces $X$ and $Y$ is denoted by $\mathcal{B}(X,Y)$\index{$\mathcal{B}(X,Y)$} and we set $\mathcal{B}(X) := \mathcal{B}(X,X)$. An operator $T:X \longrightarrow Y$ between normed spaces $X$ and $Y$ is called bounded from below by $m$ ($m>0$), if $m\Vert x \Vert_X \leq \Vert Tx \Vert_Y$ for all $x\in X$. For an operator $U \in \mathcal{B}( \mathcal{H}_1 , \mathcal{H}_2 )$ ($\mathcal{H}_1$, $\mathcal{H}_2$ Hilbert spaces) with closed range $\mathcal{R}(U)$, $U^{\dagger}$ denotes its associated \emph{pseudo inverse}\index{pseudo inverse}. Recall \cite{ole1} that the pseudo inverse of $U$ is characterized as the unique operator $U^{\dagger}: \mathcal{H}_2 \longrightarrow  \mathcal{H}_1$, satisfying the three relations 
\begin{equation}\label{pseudoinverse2}
\mathcal{N}(U^{\dagger}) = \mathcal{R}(U)^{\perp}, \, \, \, \, \mathcal{R}(U^{\dagger}) = \mathcal{N}(U)^{\perp}, \, \, \, \, UU^{\dagger}x = x \quad (x\in \mathcal{R}(U)).
\end{equation}
Moreover, if $U$ has closed range, then $U^*$ has closed range and $(U^*)^{\dagger} = (U^{\dagger})^*$. On $\mathcal{R}(U)$ we explicitly have $U^{\dagger} = U^* (UU^*)^{-1}$. In case $U$ is bounded and invertible, it holds $U^{\dagger} = U^{-1}$. 

\subsection{Frame Theory}

Recall \cite{ole1} that a \emph{frame} for $\mathcal{H}$ is a countable family $\psi = ( \psi_i )_{i\in I}$ in $\mathcal{H}$, for which there exist constants $0<A_{\psi} \leq B_{\psi} < \infty$, called the \emph{lower} and \emph{upper frame bound} respectively, such that for every $f\in \mathcal{H}$ it holds 
\begin{flalign}\label{frameinequ}
A_{\psi} \Vert f \Vert^2 \leq \sum_{i\in I} \vert \langle f, \psi_i \rangle \vert^2 \leq B_{\psi} \Vert f \Vert^2 .
\end{flalign} 
If the frame bounds $A_{\psi}$ and $B_{\psi}$ can be chosen to be equal, then $\psi$ is called an $A_{\psi}$\emph{-tight} or simply \emph{tight} frame. A $1$-tight frame is called \emph{Parseval} frame. If a sequence $\psi = ( \psi_i )_{i\in I}$ satisfies the right-hand inequality, but not necessarily the left-hand inequality in (\ref{frameinequ}), then it is called a \emph{Bessel sequence} with Bessel bound $B_{\psi}$. A sequence $\psi = ( \psi_i )_{i\in I}$ in $\mathcal{H}$ is called a \emph{Riesz basis}\index{Riesz basis}, if it is complete, i.e. $\overline{span}(\psi_i )_{i\in I} = \mathcal{H}$, and if there exist constants $0< \alpha_{\psi} \leq \beta_{\psi} < \infty $, called \emph{lower} and \emph{upper Riesz bound}\index{Riesz bound} respectively, such that for any finite scalar sequence $( c_j )_{j\in J}$ ($J\subseteq I$) we have
$$\alpha_{\psi} \sum _{j\in J} \vert c_j \vert ^2 \leq  \bigg\Vert \sum _{j\in J} c_j \psi_j \bigg\Vert ^2 \leq \beta_{\psi} \sum _{j\in J} \vert c_j \vert ^2 .$$

For an arbitrary sequence $\psi = ( \psi_i )_{i\in I}$ in $\mathcal{H}$, we consider its associated frame-related operators \cite{xxlstoeant11}:
\begin{itemize}
    \item The \emph{synthesis operator} 
    $D_{\psi} : \text{dom}(D_{\psi}) \subseteq \ell ^2 (I) \longrightarrow  \mathcal{H}$, where $\text{dom}(D_{\psi}) = \big\lbrace (c_i)_{i\in I} \in \ell^2 (I): \sum_{i\in I} c_i \psi_i \in \mathcal{H}\big\rbrace$ and $D_{\psi} ( c_i )_{i\in I} = \sum_{i\in I} c_i \psi_i$.
    \item The \emph{analysis operator} $C_{\psi} : \text{dom}(C_{\psi}) \subseteq \mathcal{H} \longrightarrow \ell ^2 (I)$, where $\text{dom}(C_{\psi}) = \lbrace f \in \mathcal{H}:  (\langle f, \psi_i \rangle )_{i\in I} \in \ell^2 (I) \rbrace$ and $C_{\psi} f = (\langle f, \psi_i \rangle )_{i \in I}$.
    \item The \emph{frame operator} $S_{\psi}: \text{dom}(S_{\psi}) \subseteq \mathcal{H} \longrightarrow \mathcal{H}$, where $\text{dom}(S_{\psi}) = \big\lbrace f \in \mathcal{H}: \sum_{i\in I} \langle f, \psi_i \rangle \psi_i \in \mathcal{H} \big\rbrace$ and $S_{\psi} f = \sum_{i\in I} \langle f, \psi_i \rangle \psi_i$.  
\end{itemize}
One can show \cite{ole1,xxlstoeant11} that a sequence $\psi = (\psi_i )_{i\in I} $ is a Bessel sequence with Bessel bound $B_{\psi}$ if and only if $\text{dom}(D_{\psi}) = \ell^2(I)$ and $D_{\psi}$ is bounded by $\sqrt{B_{\psi}}$. In that case, we also have  $D^*_{\psi} = C_{\psi} \in \mathcal{B}( \mathcal{H}, \ell^2 (I))$ and $S_{\psi} = D_{\psi}C_{\psi} \in \mathcal{B}(\mathcal{H})$ with $\Vert C_{\psi} \Vert \leq \sqrt{B_{\psi}}$ and $\Vert S_{\psi} \Vert \leq B_{\psi}$, whereas $S_{\psi}$, obviously, is self-adjoint. If $\psi$ is a frame, its associated frame operator $S_{\psi}$,  additionally, is positive and invertible, yielding the possibility of \emph{frame reconstruction} for all $f\in \mathcal{H}$ via 
\begin{flalign}\label{framerec}
f = \sum_{i\in I} \langle f, S_{\psi} ^{-1} \psi_i \rangle \psi_i = \sum_{i\in I} \langle f, \psi_i \rangle S_{\psi} ^{-1} \psi_i .
\end{flalign} 
The sequence $( \tilde{\psi}_i )_{i\in I} := (S_{\psi} ^{-1} \psi_i )_{i\in I}$ is also a frame for $\mathcal{H}$ with frame bounds $B_{\psi}^{-1} \leq A_{\psi}^{-1}$ and associated frame operator $S_{\tilde{\psi}} = S_{\psi}^{-1}$, and is referred to as the \emph{canonical dual frame} of $\psi$. More generally, a frame $\phi = (\phi_i )_{i\in I}$ is called a \emph{dual frame} of $\psi$, if 
for all $f\in \mathcal{H}$
\begin{flalign}
f = \sum_{i\in I} \langle f, \varphi_i \rangle \psi_i = \sum_{i\in I} \langle f, \psi_i \rangle \varphi_i .\notag
\end{flalign}

A frame $\psi$ is $A_{\psi}$-tight if and only if $S_{\psi}=A_{\psi} \cdot \mathcal{I}_{\mathcal{H}}$ \cite{ole1}. Therefore, frame reconstruction becomes particularly simple for tight frames, because (\ref{framerec}) then reduces to $f = \frac{1}{A_{\psi}} \sum_{i\in I} \langle f, \psi_i \rangle \psi_i$. 

More generally, if $\mathcal{H}$ is a RKHS, we say that a frame $\psi$ is \emph{painless} (compare to \emph{painless non-stationary Gabor frames} \cite{nsdgt10}), if $S_{\psi}$ is given by a multiplication operator, i.e. $S_{\psi} f (x) = m(x)f(x)$ for some function $m:X \longrightarrow \mathbb{F}$. In particular, the frame inequalities (\ref{frameinequ}) imply $A_{\psi} \leq m(x) \leq B_{\psi}$ ($x\in X$). Moreover, frame reconstruction remains 'painless' in this case, as it simply corresponds to point-wise multiplication, $S_{\psi}^{-1} f (x) = \frac{1}{m(x)} f(x)$ ($x \in X$). 

Frames and Riesz bases can be characterized in terms of their associated synthesis and analysis operators, as presented in the statement below, see also \cite{ole1}. For a more detailed study of the frame-related operators associated to general sequences, which can be unbounded operators, we refer to \cite{xxlstoeant11}. 

\newtheorem{framechar}{Theorem}[section]

\begin{framechar}\label{framechar} \cite{ole1}
Let $\psi = ( \psi_i ) _{i\in I}$ be a countable family of vectors in $\mathcal{H}$. Then the following are equivalent.
\begin{itemize}
\item[(i)] $\psi$ is a frame (resp. Riesz basis) for $\mathcal{H}$.
\item[(ii)] The synthesis operator $D_{\psi}$ is bounded and surjective (resp. bounded and bijective).
\item[(iii)] The analysis operator $C_{\psi}$ is bounded, injective and has closed range (resp. bounded and bijective).
\end{itemize}
\end{framechar}



In particular, for any frame $\psi$, the pseudo inverses of $D_{\psi}$ and $C_{\psi}$ are well-defined. In \cite{ole1} it is shown that 
\begin{flalign}
D_{\psi}^{\dagger} = C_{\psi} S_{\psi}^{-1}, \notag
\end{flalign}
which implies 
\begin{flalign}
C_{\psi}^{\dagger} = S_{\psi}^{-1} D_{\psi} .\notag
\end{flalign}
In particular, if $\psi$ is a Riesz basis, we have 
\begin{flalign}
D_{\psi}^{-1} = C_{\psi} S_{\psi}^{-1}, \quad C_{\psi}^{-1} = S_{\psi}^{-1} D_{\psi} .\notag
\end{flalign}

\subsection{Fusion Frame Theory}

\newtheorem{theorem1}{Theorem}[section]
\newtheorem{corollary}[theorem1]{Corollary}
\newtheorem{proposition1}[theorem1]{Proposition}
\newtheorem{lemma}[theorem1]{Lemma}
\newtheorem{definition1}[theorem1]{Definition}
\newtheorem{remark1}[theorem1]{Remark}
\newtheorem{fusionframechar}[framechar]{Theorem}
\newtheorem{fusionRieszbasischar}[framechar]{Theorem}

By rewriting the terms $\vert \langle f, \psi_i \rangle \vert^2$ from the frame inequalities (\ref{frameinequ}), 
\begin{flalign}
\vert \langle f,  \psi_i \rangle \vert ^2 &= \Vert \psi_i \Vert ^2 \Big\langle \langle f, \frac{\psi_i}{\Vert \psi_i \Vert} \rangle \frac{\psi_i}{\Vert \psi_i \Vert}, f \Big\rangle \notag \\
&= \Vert \psi_i \Vert ^2 \langle \pi _{V_i} f, f\rangle \notag \\
&= v_i ^2 \Vert \pi _{V_i} f \Vert ^2, \notag
\end{flalign}
where we set $V_i:= \overline{span}\lbrace \psi_i\rbrace$ and $v_i := \Vert \psi_i \Vert$, we see that frames can also be viewed as weighted sequences of ($1$-dimensional) closed subspaces (resp. orthogonal projections), satisfying the inequalities (\ref{frameinequ}). Admitting higher dimensional subspaces leads to the notion of a fusion frame and allows to extend many definitions from classical frame theory to this setting. 

Recall \cite{caskut04}, that a countable sequence $V = (V_i , v_i )_{i\in I}$ of closed subspaces $V_i$ of $\mathcal{H}$, together with weights $v_i >0$, is called a \emph{fusion frame} (or \emph{frame of subspaces}) for $\mathcal{H}$, if there exist constants $0< A_V \leq B_V < \infty$, called \emph{lower} and \emph{upper fusion frame bound} respectively, such that for every $f\in \mathcal{H}$ it holds 
\begin{flalign}\label{fusframeinequ}
A_V \Vert f \Vert^2 \leq \sum_{i\in I} v_i^2 \Vert \pi_{V_i} f \Vert^2 \leq B_V \Vert f \Vert^2 .
\end{flalign}
A fusion frame $V$ is called $A_V$\emph{-tight}, if the fusion frame bounds $A_V$ and $B_V$ can be chosen to be equal. A $1$-tight fusion frame is called \emph{Parseval} fusion frame. If $V$ as above satisfies the right-hand inequality, but not necessarily the left-hand inequality in (\ref{fusframeinequ}), we call $V$ a \emph{Bessel fusion sequence} with \emph{Bessel fusion bound} $B_V$. A weighted family of closed subspaces $V = (V_i , v_i )_{i\in I}$ is called a \emph{fusion Riesz basis}, if $\overline{span}(V_i )_{i\in I} = \mathcal{H}$ and if there exist constants $0< \alpha_V \leq \beta_V < \infty $, called \emph{lower} and \emph{upper fusion Riesz bound} respectively, such that for any finite subset $J\subseteq I$ we have 
\begin{flalign}
\alpha_V \sum _{j\in J} \Vert f_j \Vert ^2 \leq  \bigg\Vert \sum _{j\in J} v_j f_j \bigg\Vert ^2 \leq \beta_V \sum _{j\in J} \Vert f_j \Vert ^2 \notag
\end{flalign}
for all sequences $(f_j )_{j\in J} \in (V_j)_{j\in J}$. The family $(V_i)_{i\in I}$ is called an \emph{orthonormal fusion basis} for $\mathcal{H}$, if $\mathcal{H}$ is the orthogonal direct sum of the spaces $V_i$, i.e. $\mathcal{H} = \oplus_{i\in I} V_i$.

The appropriate representation spaces for the fusion frame setting are \emph{Hilbert direct sums}, defined by
\begin{flalign}
\big(\sum _{i\in I} \oplus V_i \big)_{\ell ^2} = \Big\lbrace (f_i )_{i \in I}: f_i \in V_i, \sum_{i\in I} \Vert f_i \Vert^2 < \infty \Big\rbrace . \notag
\end{flalign}

As in the classical frame setting, for any weighted sequence of closed subspaces $V=(V_i, v_i)_{i\in I}$ we consider its associated fusion-frame-related operators:
\begin{itemize}
    \item The \emph{synthesis operator} $D_V :\text{dom}(D_V) \subseteq \big(\sum _{i\in I} \oplus V_i \big)_{\ell ^2} \longrightarrow \mathcal{H}$, where $\text{dom}(D_V) = \big\lbrace (f_i)_{i\in I} \in \big(\sum _{i\in I} \oplus V_i \big)_{\ell ^2} : \sum_{i \in I} v_i f_i \in \mathcal{H} \big\rbrace$ and $D_V (f_i )_{i \in I} = \sum_{i \in I} v_i f_i$,
    \item The \emph{analysis operator} 
    $C_V : \text{dom}(C_V) \subseteq \mathcal{H} \longrightarrow \big( \sum _{i\in I} \oplus V_i \big) _{\ell ^2}$, where $\text{dom}(C_V) = \big\lbrace f \in \mathcal{H}:  ( v_i \pi _{V_i} f )_{i \in I} \in \big(\sum _{i\in I} \oplus V_i \big)_{\ell ^2} \big\rbrace$ and $C_V f = ( v_i \pi _{V_i} f )_{i \in I}$,
    \item The \emph{fusion frame operator}
    $S_V : \text{dom}(S_V) \subseteq \mathcal{H} \longrightarrow \mathcal{H}$, where $\text{dom}(S_V) = \big\lbrace f \in \mathcal{H}: \sum_{i\in I} v_i ^2 \pi_{V_i} f \in \mathcal{H} \big\rbrace$ and $S_V f = \sum_{i\in I} v_i ^2 \pi_{V_i} f$. 
\end{itemize}
It is well-known \cite{caskut04} that $V=(V_i, v_i)_{i\in I}$ is a Bessel fusion sequence with Bessel fusion bound $B_V$ if and only if $\text{dom}(D_V) = (\sum _{i\in I} \oplus V_i )_{\ell ^2}$ and $D_V$ is bounded by $\sqrt{B_V}$. Here, the basic definitions are still analogous to (Hilbert) frames, we again have $D^*_V = C_V \in \mathcal{B}(\mathcal{H}, (\sum _{i\in I} \oplus V_i )_{\ell ^2})$ and $S_V = D_V C_V \in \mathcal{B}(\mathcal{H})$ being self-adjoint. The latter is additionally positive and invertible, if $V$ is a fusion frame. In particular, fusion frames also yield the possibility of perfect reconstruction via 
\begin{flalign}
f = \sum_{i\in I} v_i ^2 \pi_{V_i} S_V ^{-1} f = \sum_{i\in I} v_i ^2 S_V ^{-1} \pi_{V_i} f.\notag
\end{flalign}
It can be shown that a fusion frame $V$ is $A_V$-tight if and only if $S_V = A_V\cdot \mathcal{I}_{\mathcal{H}}$, see \cite{kohl21} for more details. 

Moreover, if $\mathcal{H}$ is a RKHS, we call a fusion frame $V$ for $\mathcal{H}$ \emph{painless} \cite{daubgromay86,gr01,nsdgt10}, if $S_V$ is given by a multiplication operator, i.e. $S_V f (x) = m(x)f(x)$ for some function $m:X \longrightarrow \mathbb{F}$, which necessarily satisfies $A_{\psi} \leq m \leq B_{\psi}$ and $S_V^{-1} f (x) = \frac{1}{m(x)} f(x)$ ($x \in X$).

\

Unsurprisingly, fusion frames and fusion Riesz bases can be characterized in terms of their associated synthesis and analysis operators, as in the classical frame case (Theorem \ref{framechar}):
\begin{fusionframechar}\label{fusionframecar}  \cite{caskut04}
Let $V=(V_i , v_i )_{i\in I}$ be a sequence of closed subspaces in $\mathcal{H}$ with weights $v_i >0$. Then the following are equivalent.
\begin{itemize}
\item[(i)] $V$ is a fusion frame (resp. fusion Riesz basis) for $\mathcal{H}$. 
\item[(ii)] The synthesis operator $D_V$ is bounded and surjective (resp. bounded and bijective).
\item[(iii)] The analysis operator $C_V$ is bounded, injective and has closed range (resp. bounded and bijective).
\end{itemize}
\end{fusionframechar}

In particular, for any fusion frame $V$ it holds
\begin{flalign}
D_V^{\dagger} = C_V S_V^{-1}, \quad
C_V^{\dagger} = S_V^{-1} D_V .\notag
\end{flalign}
In case $V$ is a fusion Riesz basis, we have 
\begin{equation}\label{pseudoinversefusionframe2}
D_V^{-1} = C_V S_V^{-1}, \quad C_V^{-1} = S_V^{-1} D_V .
\end{equation}

For a fusion Riesz basis $V$, we can give a simpler formula for $D_V^{-1}$ than in (\ref{pseudoinversefusionframe2}), when we restrict $D_V^{-1}$ to the subspace $V_i$. By setting (for $i\in I$)
\begin{flalign}
\mathcal{V}_i := ... \times \lbrace 0 \rbrace \times \lbrace 0 \rbrace \times V_i \times \lbrace 0 \rbrace \times \lbrace 0 \rbrace \times ...  \quad \text{($V_i$ at the $i$-th position),} \notag
\end{flalign}\index{$\mathcal{V}_n$}
we can formulate the following preparatory result.

\newtheorem{inverseoperators}[framechar]{Lemma}

\begin{inverseoperators}\label{inverseoperators}
Let $V = (V_i , v_i)_{i\in I}$ be a fusion Riesz basis. Let $i\in I$ be arbitrary and $g_i \in V_i$. Then 
\begin{equation}\label{inverseVi}
D_V^{-1} g_i = (..., 0, 0, v_i^{-1}g_i, 0, 0, ...) \quad  \text{  ($v_i^{-1}g_i$ in the $i$-th entry)} \notag 
\end{equation}
In other words, $D_V^{-1}$ maps $V_i$ into $\mathcal{V}_i$ for every $i\in I$.
\end{inverseoperators}

\begin{proof}
Let $g_i \in V_i$. According to Theorem \ref{fusionframecar}, there exists precisely one $g \in \big(\sum _{i\in I} \oplus V_i \big)_{\ell ^2}$, such that $D_V g = g_i$. Obviously, $g = (..., 0, 0, v_i^{-1} g_i, 0, 0, ...)$ ($g_i$ in the $i$-th entry) meets this condition.
\end{proof}

The following result from \cite{caskut04} can be viewed as the starting point of the theory of fusion frames. It indicates the link between fusion frames and distributed processing. 

\newtheorem{start}[framechar]{Theorem}

\begin{start}\label{start}  \cite{caskut04}
Let $( V_i) _{i \in I}$ be a family of closed subspaces of $\mathcal{H}$, $(v_i)_{i \in I}$ a family of weights and assume that for every $i\in I$, $\varphi ^{(i)} := ( \varphi _{ij} )_{j\in J_i}$ is a frame for $V_i$ with frame bounds $A_i$ and $B_i$. Suppose that there exist constants $A$ and $B$ such that $0<A = \inf_{i\in I} A_i \leq \sup_{i \in I} B_i = B < \infty$. Then the following are equivalent.
\begin{itemize}
\item[(i)] $V=(V_i , v_i)_{i \in I}$ is a fusion frame for $\mathcal{H}$. 
\item[(ii)] $v\varphi = (v_i \varphi_{ij})_{i\in I, j\in J_i}$ is a frame for $\mathcal{H}$.
\end{itemize}
In particular, if $V$ is a fusion frame with fusion frame bounds $A_V$ and $B_V$, then $A_{v\varphi} = A_V\cdot A$ and $B_{v\varphi} = B_V\cdot B$ are frame bounds for $v\varphi$. Conversely, if $v\varphi$ is a frame with frame bounds $A_{v\varphi}$ and $B_{v\varphi}$, then $A_V = A_{v\varphi}/B$ and $B_V = B_{v\varphi}/A$ are fusion frame bounds for $V$. 
\end{start}


Theorem \ref{start} shows that the three frame layers -subspaces, local sequences, global sequence- are naturally connected. This motivates the notion of a \emph{fusion frame system}.

\newtheorem{fusfrsysdef}[framechar]{Definition }

\begin{fusfrsysdef} \cite{cakuli08}
Let $(V_i , v_i )_{i \in I}$ be a fusion frame for $\mathcal{H}$ and let $\varphi^{(i)} = (\varphi_{ij})_{j \in J_i}$ be a frame for $V_i$ for every $i \in I$. If the frames $\varphi^{(i)}$ have common frame bounds, then we call $(V_i , v_i , \varphi^{(i)})_{i\in I}$ a \emph{fusion frame system}\index{fusion frame system} for $\mathcal{H}$. Furthermore, we call its associated frame $v \varphi := (v_i \varphi_{ij})_{i\in I, j\in J_i}$ \emph{global frame}\index{global frame}, and the frames $\varphi^{(i)}$ \emph{local frames}\index{local frame}. We will always denote the common frame bounds of the local frames by $A$ and $B$ ($A\leq B$).
\end{fusfrsysdef}


The advantage of the concept of fusion frames and fusion frame systems is, that it enables us to reconstruct signals in two (generally different) canonical ways. We can either perform frame reconstruction via the global frame $v \varphi$, i.e.
\begin{flalign}\label{centralized}
f = \sum_{i\in I} \sum_{j\in J_i} \langle f, v_i \varphi_{ij} \rangle S_{v\varphi}^{-1} v_i \varphi_{ij} ,
\end{flalign}
or we perform frame reconstruction of $\pi_{V_i} f$ via the frame $\varphi^{(i)}$ on a local level at first and then fuse the information together via fusion frame reconstruction, i.e.
\begin{flalign}\label{distributed}
f &= \sum_{i\in I} v_i ^2 S_V^{-1} \pi_{V_i} f \notag  \\
&= \sum_{i\in I} \sum_{j\in J_i} \langle f, v_i \varphi_{ij} \rangle S_V^{-1} S_{\varphi^{(i)}}^{-1} v_i \varphi_{ij} .
\end{flalign}
Following the terminology from \cite{cakuli08}, we call the reconstruction process (\ref{centralized}) \emph{centralized reconstruction}, and the reconstruction process (\ref{distributed}) \emph{distributed reconstruction}. Here, the question, how $S_{v\varphi}^{-1}$ and the operators $S_V^{-1} S_{\varphi^{(i)}}^{-1}$ are related, naturally arises, and will be answered in Proposition \ref{dualglobal2}.

\section{Bounded block diagonal operators}\label{sec3}

In order to give the relations between the (fusion) frames associated to a fusion frame system in terms of their frame-related operators, we reconsider Hilbert direct sums and a certain type of bounded operators between them.

\

Instead of considering Hilbert direct sums of closed subspaces of the same Hilbert space $\mathcal{H}$, we consider the slightly more general case, where $V_i$ are arbitrary Hilbert spaces, and set 
\begin{flalign}\label{HDS}
\big( \sum_{i\in I} \oplus V_i \big) _{\ell ^2} := \Big\lbrace (f_i )_{i\in I} : \, f_i \in V_i , \sum_{i\in I}\Vert f_i \Vert_{V_i} ^2 <\infty \Big\rbrace .
\end{flalign}
Obviously, setting $V_i = \mathbb{C}$ for all $i\in I$ yields the space $\ell^2 (I)$.

For $f = ( f_i )_{i\in I}$, $g=  (g_i )_{i\in I} \in ( \sum_{i\in I} \oplus V_i ) _{\ell ^2}$, the operation 
\begin{flalign}\label{innerproduct}
\langle f, g\rangle _{( \sum_{i\in I} \oplus V_i ) _{\ell ^2}} := \sum_{i\in I} \langle f_i ,g_i \rangle_{V_i}
\end{flalign}
is easily seen to be well-defined and inherits the defining properties of an inner product from the inner products $\langle \cdot , \cdot \rangle_{V_i}$, see also \cite{conw1}. Its induced norm is therefore given by 
\begin{equation}\label{norm2}
\Vert f \Vert_{( \sum_{i\in I} \oplus V_i ) _{\ell ^2}} = \Big(\sum_{i\in I} \Vert f_i \Vert_{V_i}^2 \Big)^{1/2}
\end{equation}
and adapting any standard completeness proof for $\ell^2 (I)$ yields that $( \sum_{i\in I} \oplus V_i ) _{\ell ^2}$ is also complete with respect to this norm. 

Of course, definition (\ref{HDS}) extends to the case where the spaces $V_i$ are only pre-Hilbert spaces, resulting in $( \sum_{i\in I} \oplus V_i ) _{\ell ^2}$ equipped with the inner product (\ref{innerproduct}) 
being only a pre-Hilbert 
space. If and only if all $V_i$ are Hilbert spaces, the space $( \sum_{i\in I} \oplus V_i ) _{\ell ^2}$ is a Hilbert space, in which case we call it a \emph{Hilbert direct sum}. In particular, since subspaces of Hilbert spaces are closed if and only if they are complete, we obtain the following corollary.

\newtheorem{complete2}{Corollary}[section]

\begin{complete2}\label{complete2}
Consider the Hilbert direct sum $\big( \sum_{i\in I} \oplus V_i \big) _{\ell ^2}$ and for each $i\in I$, let $U_i$ be a subspace of $V_i$. Then $\big( \sum_{i\in I} \oplus U_i \big) _{\ell ^2}$ is a closed subspace of $\big( \sum_{i\in I} \oplus V_i \big) _{\ell ^2}$ if and only if $U_i$ is a closed subspace of $V_i$ for every $i\in I$. 
\end{complete2}

\newtheorem{complete3}[complete2]{Lemma}

The next result is also easy to show, but still worthwhile to be mentioned.

\begin{complete3}\label{complete3}
Consider the Hilbert direct sum $\big( \sum_{i\in I} \oplus V_i \big) _{\ell ^2}$ and the (not necessarily closed) subspaces $U_i$ of $V_i$ ($i\in I$) and $\big( \sum_{i\in I} \oplus U_i \big) _{\ell ^2}$ of $\big( \sum_{i\in I} \oplus V_i \big) _{\ell ^2}$. Then
\begin{flalign}
\bigg( \big( \sum_{i\in I} \oplus U_i \big) _{\ell ^2} \bigg) ^{\perp} = \big( \sum_{i\in I} \oplus U_i ^{\perp} \big) _{\ell ^2}. \notag
\end{flalign}
\end{complete3}

\begin{proof}
The "$\supseteq$"-part is trivial. 
To prove the "$\subseteq$"-part, we have to show that for any $h = ( h_i ) _{i\in I} \in \left( ( \sum_{i\in I} \oplus U_i ) _{\ell ^2} \right) ^{\perp}$ we have $h_i \in U_i ^{\perp}$ (for all $i\in I$). To this end, observe that since $\langle f, h \rangle_{( \sum_{i\in I} \oplus V_i ) _{\ell ^2}} = 0$  for all $f \in \big( \sum_{i\in I} \oplus U_i \big) _{\ell ^2}$, we particularly may choose $f = (..., 0, 0, f_i , 0, 0, ...) \in \big( \sum_{i\in I} \oplus U_i \big) _{\ell ^2}$ (where $f_i \in U_i$ is in the $i$-th entry) and see that in this case $0 = \langle f , h \rangle_{( \sum_{i\in I} \oplus V_i ) _{\ell ^2}} = \langle f_i , h_i \rangle _{V_i}$, which implies $h_i \in U_i ^{\perp}$. Since this holds for all $i\in I$, the proof is finished.
\end{proof}

Next we consider two Hilbert direct sums $\big( \sum_{i\in I} \oplus V_i \big) _{\ell ^2}$ and $\big( \sum_{i\in I} \oplus W_i \big) _{\ell ^2}$ indexed by the same index set. Let, for every $i\in I$, $\mathcal{O}_i \in \mathcal{B}(V_i, W_i)$ be given. We call the family $(\mathcal{O}_i)_{i\in I}$ \emph{completely bounded} \cite{paulsen_2003}, if there exists a constant $C>0$, such that $\Vert \mathcal{O}_i \Vert \leq C$ for all $i\in I$. It is easy to show the next lemma, see also \cite{conw1}.

\newtheorem{Op1}[complete2]{Lemma}

\begin{Op1}\label{Op1}
Consider the family  $( \mathcal{O}_i )_{i \in I}$ of operators $\mathcal{O}_i \in \mathcal{B}(V_i , W_i)$. Then
\begin{flalign}\label{+operator}
\bigoplus_{i\in I} \mathcal{O}_i  ( f_i )_{i\in I} := ( \mathcal{O}_i f_i )_{i\in I}
\end{flalign} defines a well-defined and bounded operator from $\left( \sum _{i\in I} \oplus V_i \right) _{\ell ^2}$ into $\left( \sum _{i\in I} \oplus W_i \right) _{\ell ^2}$ if and only if $(\mathcal{O}_i )_{i \in I}$ is completely bounded. In that case $\Vert \bigoplus_{i\in I} \mathcal{O}_i \Vert = \sup_{i\in I} \Vert \mathcal{O}_i \Vert$.
\end{Op1}

Bounded operators $\bigoplus_{i\in I} \mathcal{O}_i$, defined as in (\ref{+operator}), naturally appear in the definition of a \emph{dual fusion frame} \cite{kutpatphi17,Gav07,heimoanza14,hemo14,heimo18}. In accordance with these references we call such operators \emph{block diagonal}. We also note that in those references, block diagonal operators $\bigoplus_{i\in I} \mathcal{O}_i$ with the additional property, that each operator $\mathcal{O}_i$ is surjective, are called \emph{component preserving}. In the remainder of this section, a detailed investigation of bounded block diagonal operators between Hilbert direct sums is presented, whereas component preserving operators only implicitly appear on some occasions as a special case, see e.g. Proposition \ref{Op6}. 

We omit the obvious proof of the next result. For more details, see \cite{kohl21}.  

\newtheorem{Op2}[complete2]{Lemma}

\begin{Op2}\label{Op2}
Let $( \sum _{i\in I} \oplus U_i ) _{\ell ^2}$, $( \sum _{i\in I} \oplus V_i ) _{\ell ^2}$ and $( \sum _{i\in I} \oplus W_i ) _{\ell ^2}$ be Hilbert direct sums, let $\bigoplus_{i\in I} \mathcal{O}_i \in \mathcal{B} \big( ( \sum _{i\in I} \oplus V_i ) _{\ell ^2} , ( \sum _{i\in I} \oplus W_i ) _{\ell ^2} \big)$ and $\bigoplus_{i\in I} \mathcal{P}_i \in \mathcal{B} \big( ( \sum _{i\in I} \oplus U_i ) _{\ell ^2} , ( \sum _{i\in I} \oplus V_i ) _{\ell ^2} \big)$. Then the following hold.
\begin{itemize}
\item[(a)]
\begin{flalign}
\big( \bigoplus_{i\in I} \mathcal{O}_i \big) ^* = \bigoplus_{i\in I} \mathcal{O}^*_i  \notag
\end{flalign}
\item[(b)] If $( \sum _{i\in I} \oplus V_i ) _{\ell ^2} = ( \sum _{i\in I} \oplus W_i ) _{\ell ^2}$, then $\bigoplus_{i\in I} \mathcal{O}_i$ is self-adjoint if and only if $\mathcal{O}_i$ is self-adjoint for every $i\in I$. 
\item[(c)] 
\begin{flalign}
\big( \bigoplus_{i\in I} \mathcal{O}_i \big) \big( \bigoplus_{i\in I} \mathcal{P}_i \big) = \bigoplus_{i\in I}   (\mathcal{O}_i \mathcal{P}_i ) .\notag
\end{flalign}
\item[(d)] $\bigoplus_{i\in I} \mathcal{O}_i$ is bounded from below by $m$ if and only if $\mathcal{O}_i$ is bounded from below by $m$ for all $i\in I$.
\end{itemize}
\end{Op2}


\newtheorem{Op3}[complete2]{Lemma}

Next, we consider the kernel and range of bounded block diagonal operators, where we notice a difference in the quality of the results.

\begin{Op3}\label{Op3}
Let $\bigoplus_{i\in I} \mathcal{O}_i \in \mathcal{B} \big( ( \sum _{i\in I} \oplus V_i ) _{\ell ^2} , ( \sum _{i\in I} \oplus W_i ) _{\ell ^2} \big)$. Then
\begin{itemize}
\item[(a)] 
\begin{equation}\label{kernel}
\mathcal{N}\big( \bigoplus_{i\in I} \mathcal{O}_i \big) = \big( \sum _{i\in I} \oplus \mathcal{N}(\mathcal{O}_i ) \big) _{\ell ^2} \notag
\end{equation}
\item[(b)] 
\begin{equation}\label{range}
\mathcal{R}\big( \bigoplus_{i\in I} \mathcal{O}_i \big) \subseteq \big( \sum _{i\in I} \oplus \mathcal{R}(\mathcal{O}_i) \big) _{\ell ^2} \notag
\end{equation}
\end{itemize}
\end{Op3}

\begin{proof}
By component-wise definition (\ref{+operator}) we have
\begin{flalign}
\mathcal{N} \big( \bigoplus_{i\in I} \mathcal{O} _i \big) & = \Big\lbrace (f_i )_{i\in I} \in \big( \sum _{i\in I} \oplus V_i \big) _{\ell ^2}: \,  \mathcal{O}_i f_i = 0 \,  (\forall i\in I)\Big\rbrace \notag \\
& = \Big\lbrace ( f_i )_{i\in I} : f_i \in \mathcal{N}(\mathcal{O}_i) \, (\forall i\in I), \, \sum_{i\in I} \Vert f_i \Vert_{V_i} ^2 < \infty \Big\rbrace \notag \\
& = \big( \sum _{i\in I} \oplus \mathcal{N}(O_i) \big) _{\ell ^2}, \notag
\end{flalign}
as well as 
\begin{flalign}\label{subset}
\mathcal{R}\big( \bigoplus_{i\in I} \mathcal{O}_i \big) &= \Big\lbrace (g_i)_{i\in I} \in \big( \sum _{i\in I} \oplus W_i \big) _{\ell ^2}: \exists (f_i)_{i\in I}\in \big( \sum _{i\in I} \oplus V_i \big) _{\ell ^2}:\mathcal{O}_i f_i = g_i \, (\forall i\in I) \Big\rbrace \notag \\
&\subseteq \Big\lbrace (g_i)_{i\in I} \in \big( \sum _{i\in I} \oplus W_i \big) _{\ell ^2}: \exists (f_i)_{i\in I}\in (V_i)_{i\in I} :\mathcal{O}_i f_i = g_i \, (\forall i\in I) \Big\rbrace \\
&= \Big\lbrace (g_i)_{i\in I} \in \big( \sum _{i\in I} \oplus W_i \big) _{\ell ^2}: \, g_i \in \mathcal{R}(\mathcal{O}_i) \, (\forall i\in I) \Big\rbrace \notag \\
&= \big( \sum _{i\in I} \oplus \mathcal{R}(\mathcal{O}_i) \big) _{\ell ^2}. \notag 
\end{flalign}
\end{proof}

In general, we cannot achieve equality in (\ref{subset}) without further assumptions, except in the trivial case $\vert I\vert <\infty$ (compare to Example \ref{Ex}). However, if $\bigoplus_{i\in I} \mathcal{O}_i$ has closed range, then equality is always achieved, as the next result shows.

\newtheorem{Op4}[complete2]{Proposition}

\begin{Op4}\label{Op4}
Let $\bigoplus_{i\in I} \mathcal{O}_i \in \mathcal{B} \big( ( \sum _{i\in I} \oplus V_i ) _{\ell ^2} , ( \sum _{i\in I} \oplus W_i ) _{\ell ^2} \big)$ and assume that $\bigoplus_{i\in I} \mathcal{O}_i$ has closed range. Then
\begin{itemize}
\item[(a)] \begin{flalign}
    \mathcal{R}\big( \bigoplus_{i\in I} \mathcal{O}_i \big) = \big( \sum _{i\in I} \oplus \mathcal{R}(\mathcal{O}_i) \big) _{\ell ^2} .\notag
\end{flalign}
\item[(b)] $\mathcal{O}_i$ has closed range for every $i\in I$.
\end{itemize}
\end{Op4}

\begin{proof}
\noindent (a) We have
\begin{flalign}
    \overline{\big( \sum_{i\in I} \oplus \mathcal{R}(\mathcal{O}_i ) \big) _{\ell ^2}} &= \left( \left( \big( \sum_{i\in I} \oplus \mathcal{R}(\mathcal{O}_i ) \big) _{\ell ^2} \right)^{\perp} \right) ^{\perp}  \notag \\
    &= \left( \big( \sum_{i\in I} \oplus \mathcal{R}(\mathcal{O}_i ) ^{\perp} \big) _{\ell ^2} \right)^{\perp} && \llap{(by Lemma \ref{complete3})} \notag \\
    &= \left( \big( \sum_{i\in I} \oplus \mathcal{N}(\mathcal{O}_i ^* ) \big) _{\ell ^2} \right)^{\perp} \notag \\
    &= \mathcal{N}\Big( \bigoplus_{i\in I} \mathcal{O}_i ^*  \Big)^{\perp} && \llap{(by Lemma \ref{Op3} (a))} \notag \\
    &= \mathcal{N}\bigg( \Big( \bigoplus_{i\in I} \mathcal{O}_i \Big) ^* \bigg)^{\perp} && \llap{(by Lemma \ref{Op2} (a))} \notag \\
    &= \left( \mathcal{R}\Big( \bigoplus_{i\in I} \mathcal{O}_i \Big) ^{\perp} \right)^{\perp} \notag \\
    &= \overline{\mathcal{R}\big( \bigoplus_{i\in I} \mathcal{O}_i \big)}  \notag \\
     &= \mathcal{R}\big( \bigoplus_{i\in I} \mathcal{O}_i \big) \notag \\
     &\subseteq \big( \sum_{i\in I} \oplus \mathcal{R}(\mathcal{O}_i ) \big) _{\ell ^2} && \llap{(by Lemma \ref{Op3} (b))} , \notag 
\end{flalign}
which implies (a) and also shows that $\big( \sum_{i\in I} \oplus \mathcal{R}(\mathcal{O}_i ) \big) _{\ell ^2}$ is a closed subspace of $\big( \sum_{i\in I} \oplus W_i \big) _{\ell ^2}$. By Corollary \ref{complete2}, this implies (b).
\end{proof}

Proposition \ref{Op4} shows, that if $\bigoplus_{i\in I} \mathcal{O}_i$ has closed range, so do the operators $\mathcal{O}_i$ for all $i\in I$. Therefore their corresponding pseudo inverses $\big(\bigoplus_{i\in I} \mathcal{O}_i \big) ^{\dagger}$ and $\mathcal{O}_i ^{\dagger}$ ($i\in I$) are well-defined. In the next result, we give a relation between them.

\newtheorem{Op5}[complete2]{Proposition}

\begin{Op5}\label{Op5}
Let $\bigoplus_{i\in I} \mathcal{O}_i \in \mathcal{B} \Big( \big( \sum _{i\in I} \oplus V_i \big) _{\ell ^2} , \big( \sum _{i\in I} \oplus W_i \big) _{\ell ^2} \Big)$ and assume that $\bigoplus_{i\in I} \mathcal{O}_i$ has closed range. If, in addition, the family $( \mathcal{O}_i ^{\dagger} )_{i\in I}$ is completely bounded and if $\bigoplus_{i\in I} \mathcal{O}_i ^{\dagger}$ has closed range, then 
\begin{flalign}
\Big( \bigoplus_{i\in I} \mathcal{O} _i \Big)^{\dagger} = \bigoplus_{i\in I} \mathcal{O}_i ^{\dagger} .\notag
\end{flalign}
\end{Op5}

\begin{proof}
By Lemma \ref{Op1}, the operators 
$$\Big( \bigoplus_{i\in I} \mathcal{O} _i \Big)^{\dagger} : \big( \sum _{i\in I} \oplus W_i \big) _{\ell ^2} \longrightarrow \big( \sum _{i\in I} \oplus V_i \big) _{\ell ^2}$$
and
$$\bigoplus_{i\in I} \mathcal{O} _i ^{\dagger} : \big( \sum _{i\in I} \oplus W_i \big) _{\ell ^2} \longrightarrow \big( \sum _{i\in I} \oplus V_i \big) _{\ell ^2}$$ 
both are well-defined and bounded. Moreover, it holds 
$$\big( \bigoplus_{i\in I} \mathcal{O} _i \big) \big( \bigoplus_{i\in I} \mathcal{O} _i ^{\dagger} \big) \big( \bigoplus_{i\in I} \mathcal{O} _i \big) = \bigoplus_{i\in I} (\mathcal{O} _i \mathcal{O} _i ^{\dagger} \mathcal{O} _i ) = \bigoplus_{i\in I} \mathcal{O} _i ,$$ 
as well as
\begin{flalign}
    \mathcal{N} \big( \bigoplus_{i\in I} \mathcal{O} _i ^{\dagger} \big) &= \big( \sum _{i\in I} \oplus \mathcal{N}(\mathcal{O} _i ^{\dagger}) \big) _{\ell ^2} &&\llap{(by Lemma \ref{Op3} (a))} \notag \\
    &= \big( \sum _{i\in I} \oplus (\mathcal{R}(\mathcal{O} _i ) ^{\perp} \big) _{\ell ^2} &&\llap{(by (\ref{pseudoinverse2}))} \notag \\
    &= \big( \sum _{i\in I} \oplus \mathcal{R}(\mathcal{O} _i ) \big) _{\ell ^2} ^{\perp} &&\llap{(by Lemma \ref{complete3})} \notag \\
    &= \mathcal{R} \big( \bigoplus_{i\in I}\mathcal{O} _i \big)^{\perp} &&\llap{(by Proposition \ref{Op4} (a))} . \notag
\end{flalign}
Finally, the assumption, that $\bigoplus_{i \in I} \mathcal{O}_i ^{\dagger}$ has closed range, guarantees via Proposition \ref{Op4} (a) that 
$$\mathcal{R} \big( \bigoplus_{i\in I} \mathcal{O} _i ^{\dagger} \big) = \big( \sum _{i\in I} \oplus \mathcal{R}(\mathcal{O} _i ^{\dagger}) \big) _{\ell ^2} ,$$ hence
\begin{flalign}
\mathcal{R} \big( \bigoplus_{i\in I} \mathcal{O} _i ^{\dagger} \big) = \big( \sum _{i\in I} \oplus \mathcal{R}(\mathcal{O} _i ^{\dagger}) \big) _{\ell ^2} &= \big( \sum _{i\in I} \oplus (\mathcal{N}(\mathcal{O} _i ) ^{\perp} \big) _{\ell ^2} &&\llap{(by (\ref{pseudoinverse2}))} \notag \\
&= \big( \sum _{i\in I} \oplus \mathcal{N}(\mathcal{O} _i ) \big) _{\ell ^2} ^{\perp} &&\llap{(by Lemma \ref{complete3})} \notag \\
&= \mathcal{N} \big( \bigoplus_{i\in I}\mathcal{O} _i \big)^{\perp} &&\llap{(by Lemma \ref{Op3} (a))} . \notag
\end{flalign}
\end{proof}

Next we discuss the properties injectivity, surjectivity and invertibility of block diagonal operators between Hilbert direct sums.

\newtheorem{Op6}[complete2]{Proposition}

\begin{Op6}\label{Op6}
Let $\bigoplus_{i\in I} \mathcal{O}_i \in \mathcal{B} \Big( \big( \sum _{i\in I} \oplus V_i \big) _{\ell ^2} , \big( \sum _{i\in I} \oplus W_i \big) _{\ell ^2} \Big)$. Then
\begin{itemize}
\item[(a)] 
$\bigoplus_{i\in I} \mathcal{O}_i$ is injective if and only if $\mathcal{O}_i$ is injective for all $i\in I$. 
\item[(b)] If $\bigoplus_{i\in I} \mathcal{O}_i$ is surjective, then $\mathcal{O}_i$ is surjective  for all $i\in I$. 
\item[(c)] If $\mathcal{O}_i$ is surjective for all $i\in I$ and if the family $(\mathcal{O}_i ^{\dagger} )_{i\in I}$ is completely bounded, then $\bigoplus_{i\in I} \mathcal{O}_i$ is surjective.
\item[(d)] If $\bigoplus_{i\in I} \mathcal{O}_i$ is invertible, then $\mathcal{O}_i$ is invertible for all $i\in I$ and 
\begin{equation}\label{+invertible}
\Big( \bigoplus_{i\in I} \mathcal{O}_i \Big) ^{-1} = \bigoplus_{i\in I} \mathcal{O}_i ^{-1} .
\end{equation}
\item[(e)] If $\mathcal{O}_i$ is invertible for all $i\in I$ and if $(\mathcal{O}_i ^{-1} )_{i\in I}$ is completely bounded, then also $\bigoplus_{i\in I} \mathcal{O}_i $ is invertible and (\ref{+invertible}) holds. 
\end{itemize}
\end{Op6}

\begin{proof}
\noindent (a) follows immediately from Lemma \ref{Op3} (a).

\noindent (b) Assume that $\bigoplus_{i\in I} \mathcal{O}_i$ is surjective and choose $i\in I$ arbitrary. Then for any $g_i \in W_i$, $(...,0, 0, g_i, 0, 0, ...) \in \big( \sum _{i\in I} \oplus W_i \big) _{\ell ^2}$ and by assumption there exists some $(f_i )_{i\in I} \in \big( \sum _{i\in I} \oplus V_i \big) _{\ell ^2}$ such that $\bigoplus_{i\in I} \mathcal{O}_i  ( f_i)_{i\in I}  = (...,0, 0, g_i, 0, 0, ...)$. This implies $\mathcal{O}_i f_i = g_i$, hence $\mathcal{O}_i$ is surjective.

\noindent (c) We have $\bigoplus_{i\in I} \mathcal{O}_i ^{\dagger} \in \mathcal{B}\big( (\sum _{i\in I} \oplus W_i ) _{\ell ^2}, ( \sum _{i\in I} \oplus V_i ) _{\ell ^2} \big)$ by Lemma \ref{Op1}. Moreover, by properties of pseudo-inverses, $\mathcal{O}_i ^{\dagger}$ is a right-inverse of $\mathcal{O}_i$ on $\mathcal{R}(\mathcal{O}_i) = W_i$ ($i\in I$). Therefore, by component-wise definition and Lemma \ref{Op2} (c), $\bigoplus_{i\in I} \mathcal{O}_i ^{\dagger}$ is a right-inverse of $\bigoplus_{i\in I} \mathcal{O}_i$ on $\big( \sum _{i\in I} \bigoplus W_i \big) _{\ell ^2}$. In particular, for any $(g_i)_{i\in I} \in \big( \sum _{i\in I} \bigoplus W_i \big) _{\ell ^2}$, we can find $(f_i)_{i\in I} := (\mathcal{O}_i ^{\dagger} g_i )_{i\in I} \in \big( \sum _{i\in I} \bigoplus V_i \big) _{\ell ^2}$, such that $\bigoplus_{i\in I} \mathcal{O}_i  (f_i)_{i\in I} = (g_i)_{i\in I}$, i.e. $ \bigoplus_{i\in I} \mathcal{O}_i$ is surjective.

\noindent (d) By (a) and (b), $\mathcal{O}_i$ is bijective for every $i\in I$. Moreover, since $\bigoplus_{i\in I} \mathcal{O}_i$ is invertible, it is bounded from below by some constant $m>0$, which implies that $\mathcal{O}_i$ is bounded from below by $m$ for all $i\in I$. Consequently, $(\mathcal{O}_i^{-1})_{i\in I}$ is completely bounded by $m^{-1}$, hence $\bigoplus_{i\in I} \mathcal{O}_i^{-1}$ is well-defined and bounded and obviously a left- and right-inverse of $\bigoplus_{i\in I} \mathcal{O}_i$. 

\noindent (e) Since $(\mathcal{O}_i ^{-1} )_{i\in I} = (\mathcal{O}_i ^{\dagger})_{i\in I}$ is completely bounded, we may apply (a) and (c) to see that $\bigoplus_{i\in I} \mathcal{O}_i$ is bijective and thus invertible. Moreover $\bigoplus_{i\in I} \mathcal{O}_i ^{-1}$ is well-defined and bounded and (\ref{+invertible}) clearly holds.
\end{proof}

\newtheorem{Op7}[complete2]{Corollary}


\newtheorem{Ex}[complete2]{Example}
\begin{Ex}\label{Ex}
Let $I = \mathbb{N}$ and let $V_i = W_i \neq \lbrace 0 \rbrace$ for all $i \in I$. For every $i \in I$, we define 
$$\mathcal{O}_i : V_i \longrightarrow V_i , \, \, \, \mathcal{O}_i := \frac{1}{\sqrt{i}} \mathcal{I}_{V_i}.$$
For each $i$ we have $\Vert \mathcal{O}_i \Vert \leq 1$, i.e. the family $(\mathcal{O}_i )_{i\in I}$ is completely bounded by $1$. Thus, by Lemma \ref{Op1}, $ \bigoplus_{i\in I} \mathcal{O}_i : \big( \sum _{i\in I} \oplus V_i \big) _{\ell ^2} \longrightarrow \big( \sum _{i\in I} \oplus V_i \big) _{\ell ^2}$ is a well-defined and bounded operator.
\begin{itemize}
\item[(a)] Let us reconsider Lemma \ref{Op3} (b). We give an example of an element $g\in \big( \sum _{i\in I} \oplus V_i \big) _{\ell ^2}$, which is contained in $\big( \sum _{i\in I} \oplus \mathcal{R}(\mathcal{O}_i) \big)_{\ell ^2}$ but not in $\mathcal{R}\big( \bigoplus_{i\in I} \mathcal{O}_i \big)$: For every $i\in I$, choose some normalized vector $h_i \in V_i$ and set $f_i := \frac{h_i}{\sqrt{i}} \in V_i$ and $g_i := \frac{h_i}{i} \in V_i$. Clearly we have $\mathcal{O}_i f_i = g_i$ for each $i\in I$. Observe that $g = ( g_i )_{i\in I} \in \left( \sum _{i\in I} \bigoplus V_i \right) _{\ell ^2}$, since $\Vert g \Vert ^2 _{( \sum _{i\in I} \oplus V_i ) _{\ell ^2}} = \sum_{i\in \mathbb{N}} \frac{1}{i^2} = \frac{\pi ^2}{6}$. In particular, $g \in \big( \sum _{i\in I} \oplus \mathcal{R}(\mathcal{O}_i) \big) _{\ell ^2}$. On the other hand, $f = ( f_i )_{i\in I} \not\in ( \sum _{i\in I} \bigoplus V_i ) _{\ell ^2}$, since $\Vert f \Vert ^2 _{( \sum _{i\in I} \oplus V_i ) _{\ell ^2}} = \sum_{i\in \mathbb{N}} \frac{1}{i} = \infty$.  However, $g \not\in \mathcal{R}\big( \bigoplus_{i\in I} \mathcal{O}_i \big)$, since, by component-wise definition, $f$ is the only possible candidate to be mapped  onto $g$ by $\bigoplus_{i\in I} \mathcal{O}_i$, while $f \not\in \big( \sum _{i\in I} \oplus V_i \big) _{\ell ^2}$.
\item[(b)] We also reconsider Proposition \ref{Op6} (c). At first glance, one might guess that if $( \mathcal{O}_i )_{i\in I}$ is a completely bounded family of surjective operators, then $ \bigoplus_{i\in I} \mathcal{O}_i$ has to be surjective as well. However, the above example demonstrates the importance of the (in this case missing) condition that $( \mathcal{O}_i ^{\dagger} ) _{i\in I}$ is completely bounded: The operators $O_i = \frac{1}{\sqrt{i}} \mathcal{I}_{V_i}$ are not only surjective and completely bounded, but also injective. However, $\bigoplus_{i\in I} \mathcal{O}_i$ is not surjective as shown above. Observe that $\Vert\mathcal{O}_i  ^{\dagger} \Vert = \Vert \mathcal{O}_i ^{-1} \Vert = \sqrt{i}$, i.e. the family $( \mathcal{O}_i ^{\dagger} ) _{i\in I}$ fails to be completely bounded. 
\end{itemize}
\end{Ex}

\section{Fusion frame systems and related operators}\label{sec4}

We are now prepared to prove relations between the (fusion) frames associated to a fusion
frame system in terms of their associated synthesis, analysis and frame operators.

\subsection{Operator identities for fusion frame systems}\label{subsec4-2}

Let $(V_i , v_i , \varphi ^{(i)})_{i \in I}$ be a fusion frame system with corresponding global frame $v\varphi = (v_i \varphi_{ij})_{i\in I, j\in J_i}$. The representation space of $v\varphi$ is the space
$$\ell ^2 \big( \biguplus _{i\in I} J_i \big) := \Big\lbrace (c_{ij})_{i\in I, j\in J_i}: c_{ij}\in \mathbb{C}, \sum_{i\in I} \sum_{j\in J_i} \vert c_{ij}\vert^2 < \infty \Big\rbrace ,$$
which clearly is a Hilbert space with respect to the inner product $\langle (c_{ij})_{i\in I, j\in J_i}, (d_{ij})_{i\in I, j\in J_i} \rangle = \sum_{i\in I} \sum_{j\in J_i} c_{ij}\overline{d_{ij}}$. We observe that norm of $(c_{ij})_{i\in I, j\in J_i} \in \ell ^2 ( \biguplus _{i\in I} J_i )$, which is given by  
$$\Vert (c_{ij})_{i\in I, j\in J_i}\Vert = \big( \sum_{i\in I} \sum_{j\in J_i} \vert c_{ij}\vert^2 \big)^{1/2},$$ 
equals the norm of $(c_i)_{i\in I} \in \big( \sum _{i\in I}\oplus \ell ^2 (J_i) \big) _{\ell ^2}$, where $c_i = (c_{ij})_{j\in J_i} \in \ell^2(J_i)$. In particular, we obtain the following result, see \cite{kohl21} for more details.

\newtheorem{lemmaisometry}{Proposition}[section]

\begin{lemmaisometry}\label{lemmaisometry}
The Hilbert spaces $\ell ^2 \big( \biguplus _{i\in I} J_i \big)$ and $\big( \sum _{i\in I}\oplus \ell ^2 (J_i) \big) _{\ell ^2}$ are isometrically isomorphic. 
\end{lemmaisometry}
%

Recall, that the definition of a fusion frame system implies that the family $(D_{\varphi^{(i)}})_{i\in I}$ is completely bounded by $\sqrt{B}$. This observation begs for an application of our results from Section \ref{sec3}. Indeed, the link between the global frame $v \varphi = (v_i \varphi_{ij} )_{i\in I, j\in J_i}$, the fusion frame $V=(V_i , v_i) _{i\in I}$ and the local frames $\varphi ^{(i)} = ( \varphi _{ij} )_{j\in J_i}$ corresponding to the fusion frame system $(V_i , v_i , \varphi^{(i)})_{i\in I}$ is mirrored by the operator identities below. We note that identities (\ref{synthesis}) and (\ref{analysis}) also appear in \cite{heimo18}, their finite-dimensional versions can be found in \cite{hemo14}. However, for the convenience of the reader, we state a full proof.    

\newtheorem{great}[lemmaisometry]{Proposition}

\begin{great} \label{great}
Let $(V_i , v_i , \varphi ^{(i)} )_{i \in I}$ be a fusion frame system and $v\varphi$ be the corresponding global frame. Then the following hold:
\begin{flalign}
D_{v\varphi} &= D_V  \bigoplus _{i\in I} D_{\varphi ^{(i)}} \label{synthesis} \\
C_{v\varphi} &= \big( \bigoplus _{i\in I} C_{\varphi ^{(i)}} \big) C_V \label{analysis} \\
S_{v\varphi} &= D_V \big( \bigoplus _{i\in I} S_{\varphi ^{(i)}} \big) C_V .\label{frame}
\end{flalign}
\end{great}

\begin{proof}
If $ c = (c_{ij})_{i \in I, j \in J_i} \in \ell ^2 (\biguplus_{i\in I} J_i)  \cong \left( \sum_{i\in I}\oplus \ell ^2 (J_i) \right) _{l^2}$, then $c_i = (c_{ij} )_{j \in J_i} \in \ell ^2 (J_i)$ for every $i\in I$. Since the family $(D_{\varphi ^{(i)}})_{i\in I}$ is completely bounded by $\sqrt{B}$, Lemma \ref{Op1} guarantees that $\bigoplus_{i\in I} D_{\varphi ^{(i)}}$ is well-defined and bounded, i.e. that
$$(D_{\varphi ^{(i)}} c_i ) _{ i\in I}\in \big(\sum_{i\in I}\oplus V_i \big)_{\ell^2} .$$
Therefore we may write
$$D_{v\varphi} c = \sum_{i\in I} \sum_{j\in J_i} c_{ij} v_i \varphi_{ij} = \sum_{i\in I} v_i D_{\varphi ^{(i)}} c_i = D_V \bigoplus _{i\in I} D_{\varphi ^{(i)}} c ,$$
which shows \eqref{synthesis}. Applying Lemma \ref{Op2} yields
\begin{flalign}
C_{v\varphi} = D^*_{v\varphi} &= \Big( D_V \bigoplus _{i\in I} D_{\varphi ^{(i)}} \Big) ^* \notag \\
&= \Big( \bigoplus _{i\in I} D_{\varphi ^{(i)}} \Big) ^* D_V^* \notag \\ 
&= \Big( \bigoplus _{i\in I} D_{\varphi ^{(i)}}^* \Big) D_V^* = \Big( \bigoplus _{i\in I} C_{\varphi ^{(i)}} \Big) C_V \notag
\end{flalign}
and 
\begin{flalign}
S_{v\varphi} = D_{v\varphi} C_{v\varphi} &= D_V \Big( \bigoplus _{i\in I} D_{\varphi ^{(i)}} \Big) \Big( \bigoplus _{i\in I} C_{\varphi ^{(i)}} \Big) C_V  \notag \\
&= D_V \Big( \bigoplus _{i\in I} D_{\varphi ^{(i)}} C_{\varphi ^{(i)}} \Big) C_V = D_V \Big( \bigoplus _{i\in I} S_{\varphi ^{(i)}} \Big) C_V .\notag  
\end{flalign}
\end{proof}

For fusion Riesz bases, Proposition \ref{great}, Theorem \ref{fusionframecar} and Proposition \ref{Op6} (e) imply the following:

\newtheorem{dualglobal}[lemmaisometry]{Proposition}

\begin{dualglobal}\label{dualglobal}
Let $(V_i , v_i , \varphi ^{(i)} )_{i \in I}$ be a fusion frame system and $v\varphi$ be its corresponding global frame. If in addition $(V_i , v_i ) _{i \in I}$ is a fusion Riesz basis, then 
\begin{equation}\label{Svphi-1}
S_{v\varphi}^{-1} = C_V^{-1} \big( \bigoplus _{i\in I} S^{-1}_{\varphi ^{(i)}} \big) D^{-1}_V .
\end{equation}
\end{dualglobal}


Considering the previous result leads to the question, whether we can replace $C_V^{-1}$ and $D_V^{-1}$ in (\ref{Svphi-1}) by $C_V^{\dagger}$ and $D_V^{\dagger}$ respectively, and obtain a generalization of Proposition \ref{dualglobal} to fusion frame systems, where the associated fusion frame is not necessarily a fusion Riesz basis. A positive answer to this question can be given by adding a technical assumption:

\newtheorem{dualglobal2}[lemmaisometry]{Proposition}

\begin{dualglobal2}\label{dualglobal2}
Let $(V_i , v_i , \varphi ^{(i)} )_{i \in I}$ be a fusion frame system and $v\varphi$ be the corresponding global frame. If 
\begin{equation}\label{ass}
\pi_{V_i} S_V^{-1} S_{v\varphi} = S_{\varphi^{(i)}} \pi_{V_i} \quad (\forall i\in I),
\end{equation}
then 
\begin{flalign}
S_{v\varphi}^{-1} = C_V^{\dagger} \big( \bigoplus _{i\in I} S^{-1}_{\varphi ^{(i)}} \big) D_V^{\dagger}.\notag
\end{flalign}
\end{dualglobal2}

\begin{proof}
For every $f\in \mathcal{H}$ we have
\begin{flalign}
C_V^{\dagger} \big( \bigoplus _{i\in I} S^{-1}_{\varphi ^{(i)}} \big) D_V^{\dagger} S_{v\varphi} f &= S_V^{-1} D_V \big( \bigoplus _{i\in I} S^{-1}_{\varphi ^{(i)}} \big) C_V S_V^{-1} S_{v\varphi} f \notag \\
&= S_V ^{-1} \sum_{i\in I} S^{-1}_{\varphi ^{(i)}} v_i ^2 \pi_{V_i} S_V^{-1} S_{v\varphi} f  \notag \\
&= S_V ^{-1} \sum_{i\in I} S^{-1}_{\varphi ^{(i)}} v_i ^2 S_{\varphi^{(i)}} \pi_{V_i} f &&\llap{(by (\ref{ass}))} \notag \\
&= S_V ^{-1} \sum_{i\in I} v_i ^2 \pi_{V_i} f \notag \\
&= S_V ^{-1} S_V f = f. \notag
\end{flalign}
This yields the claim.
\end{proof}

\newtheorem{Ex2}[lemmaisometry]{Example}

\begin{Ex2}
(a) We show that if the fusion frame associated to a fusion frame system is a fusion Riesz basis, then (\ref{ass}) is satisfied: Let $(V_i , v_i , \varphi ^{(i)} )_{i \in I}$ be a fusion frame system with corresponding global frame $v\varphi$. It can be shown \cite{Shaarbal} that the associated fusion frame $V$ is a fusion Riesz basis if and only if $v_i^2 \pi_{V_i} S_V^{-1} \pi_{V_j} = \delta_{ij} \pi_{V_j}$ for all $i,j \in I$. The latter implies (\ref{ass}), since by Proposition \ref{great}, for all $f\in \mathcal{H}$ we have
\begin{flalign}
\pi_{V_i} S_V^{-1} S_{v\varphi} f &= \pi_{V_i} S_V^{-1} D_V \big( \bigoplus _{i\in I} S_{\varphi ^{(i)}} \big) C_V f \notag \\
&= v_i ^2 \pi_{V_i} S_V^{-1} \Big( v_i^{-2} \sum_{j\in I} \pi_{V_j} v_j ^2 S_{\varphi ^{(j)}} \pi_{V_j} f \Big) \notag \\
&= \pi_{V_i} S_{\varphi ^{(i)}} \pi_{V_i} f = S_{\varphi ^{(i)}} \pi_{V_i} f.  \notag
\end{flalign}

\noindent (b) We give an example of a fusion frame system fulfilling (\ref{ass}), while its associated fusion frame is not a fusion Riesz basis: Let $(V_i , v_i , \varphi ^{(i)} )_{i \in I}$ be a fusion frame system with corresponding global frame $v\varphi$, and assume that all local frames $\varphi ^{(i)}$ are tight with respect to the same frame bound $A$. Then, by Corollary \ref{structures}, this implies $S_{v\varphi} = A\cdot S_V$. Therefore, for all $i\in I$ we have 
$$\pi_{V_i} S_V^{-1} S_{v\varphi} = \pi_{V_i} \cdot A = A \cdot \pi_{V_i} = S_{\varphi^{(i)}} \pi_{V_i} .$$
Note that the associated fusion frame $V$ does not need to be a fusion Riesz basis. 

\noindent 
\end{Ex2}

\newtheorem{pseudolemma}[lemmaisometry]{Lemma}

\begin{pseudolemma}\label{pseudolemma}
Let $(V_i , v_i , \varphi ^{(i)} )_{i \in I}$ be a fusion frame system. Then $\bigoplus_{i\in I} D_{\varphi^{(i)}}$ is surjective and
\begin{equation}\label{pseudoA}
    \Big( \bigoplus_{i\in I} D_{\varphi^{(i)}} \Big)^{\dagger} =  \bigoplus_{i\in I} D_{\varphi^{(i)}}^{\dagger} = \bigoplus_{i\in I} C_{\widetilde{\varphi^{(i)}}}.
\end{equation}
If all local frames $\varphi ^{(i)}$ are Riesz bases, then $\bigoplus_{i\in I} D_{\varphi^{(i)}}$ is invertible and it holds 
\begin{flalign}
    \Big( \bigoplus_{i\in I} D_{\varphi^{(i)}} \Big)^{-1} =  \bigoplus_{i\in I} D_{\varphi^{(i)}}^{-1} = \bigoplus_{i\in I} C_{\widetilde{\varphi^{(i)}}}.\notag
\end{flalign}
\end{pseudolemma}

\begin{proof}
By Theorem \ref{framechar}, each $D_{\varphi^{(i)}}$ is surjective. Moreover, the family $(D_{\varphi^{(i)}}^{\dagger})_{i\in I} = (C_{\varphi ^{(i)}}S_{\varphi ^{(i)}}^{-1})_{i\in I}$ is completely bounded by $\sqrt{B}/A$. By Proposition \ref{Op6} (c), this implies that $\bigoplus_{i\in I} D_{\varphi^{(i)}}$ is surjective. Therefore, $\big( \bigoplus_{i\in I} D_{\varphi^{(i)}}\big)^* = \bigoplus_{i\in I} D_{\varphi^{(i)}}^* = \bigoplus_{i\in I} C_{\varphi^{(i)}}$ has closed range. In particular  
\begin{flalign}
\bigoplus_{i\in I} D_{\varphi^{(i)}}^{\dagger} &= \bigoplus_{i\in I} (C_{\varphi ^{(i)}}S_{\varphi ^{(i)}}^{-1}) \notag\\
&= \Big( \bigoplus_{i\in I} C_{\varphi ^{(i)}}\Big) \Big( \bigoplus_{i\in I} S_{\varphi ^{(i)}}^{-1}\Big) \notag
\end{flalign}
has closed range, since it is a composition of bounded closed range operators. Now, an application of Proposition \ref{Op5} yields \eqref{pseudoA}. The second statement can be shown similarly by applying Proposition \ref{Op6} (e).
\end{proof}

As a consequence of the previous lemma, we obtain operator identities for the fusion frame-related operators of a fusion frame system, similar to those of Proposition \ref{great}. By multiplying identity (\ref{synthesis}) from the right with $\big( \bigoplus_{i\in I} D_{\varphi^{(i)}} \big)^{\dagger} = \bigoplus_{i\in I} (C_{\varphi^{(i)}} S_{\varphi^{(i)}}^{-1})$, we obtain \eqref{fusionsynthesis} (see also \cite{heimo18,hemo14}). Proceeding analogously to the proof of Proposition \ref{great}, we immediately obtain (\ref{fusionanalysis}) and (\ref{fusionframe}): 

\newtheorem{fusionoperators}[lemmaisometry]{Proposition}

\begin{fusionoperators}\label{fusionoperators}
Let $(V_i , v_i , \varphi ^{(i)} )_{i \in I}$ be a fusion frame system and $v\varphi$ its corresponding global frame. Then 
\begin{flalign}
D_V &= D_{v\varphi} \bigoplus _{i\in I} (C_{\varphi ^{(i)} }S_{\varphi ^{(i)}}^{-1} ) \label{fusionsynthesis} \\
C_V &= \bigoplus _{i\in I} (S_{\varphi ^{(i)}}^{-1} D_{\varphi ^{(i)}} ) C_{v\varphi} \label{fusionanalysis} \\
S_V &= D_{v\varphi} \bigoplus _{i\in I} (C_{\varphi ^{(i)}} S_{\varphi ^{(i)}}^{-2} D_{\varphi ^{(i)}} ) C_{v\varphi} . \label{fusionframe} 
\end{flalign}
\end{fusionoperators}


\subsection{Properties preserved in fusion frame systems}

Our results from Section \ref{subsec4-2} enable us to examine fusion frame systems in terms of their associated frame-related operators. This is particularly interesting, since the properties of these operators are directly linked to the properties of their respective frames. In this spirit, we prove the following characterization by using operator theoretic arguments. We remark that the implication (ii) $\Rightarrow$ (i) has been independently proved in \cite{Shaarbal} via another approach.    

\newtheorem{great2}[lemmaisometry]{Theorem}

\begin{great2}\label{great2}
Let $(V_i , v_i , \varphi ^{(i)} )_{i \in I}$ be a fusion frame system and $v\varphi$ be the corresponding global frame. Then the following are equivalent: 
\begin{itemize}
\item[(i)] $v \varphi$ is a Riesz basis. 
\item[(ii)] $(V_i , v_i )_{i \in I}$ is a fusion Riesz basis and $\varphi ^{(i)}$ is a Riesz basis for every $i\in I$.
\end{itemize}
In particular, if $V=(V_i , v_i )_{i \in I}$ is a fusion Riesz basis with fusion Riesz bounds $\alpha_V$ and $\beta_V$, then $\alpha_{v\varphi} = \alpha_V \cdot A$ and $\beta_{v\varphi} = \beta_V \cdot B$ are Riesz bounds for $v\varphi$. Conversely, if $v\varphi$ is a Riesz basis with Riesz bounds $\alpha_{v\varphi}$ and $\beta_{v\varphi}$, then $\alpha_V = A_{v\varphi}/B$ and $\beta_V = B_{v\varphi}/A$ are fusion Riesz bounds for $V$. 
\end{great2}

\begin{proof}
Hereinafter, we implicitly apply Theorems \ref{framechar} and \ref{fusionframecar} several times.

(i) $\Rightarrow$ (ii): If $v \varphi$ is a Riesz basis, then $D_{v\varphi} = D_V \bigoplus _{i\in I} D_{\varphi ^{(i)}}$ is bounded and bijective. In particular,  $\bigoplus _{i\in I} D_{\varphi ^{(i)}}$ is injective, which implies by Proposition \ref{Op6} that $D_{\varphi ^{(i)}}$ is bounded and bijective for every $i\in I$. This is equivalent to $\varphi ^{(i)}$ being a Riesz basis for every $i\in I$. Moreover, by Proposition \ref{fusionoperators}, we now have that $D_V = D_{v\varphi} \bigoplus _{i\in I} (C_{\varphi ^{(i)} }S_{\varphi ^{(i)}}^{-1} )$ is a composition of bounded and bijective operators and thus itself is  bounded and bijective. Therefore, $(V_i , v_i )_{i \in I}$ is a fusion Riesz basis.

(ii) $\Rightarrow$ (i): If all $\varphi ^{(i)}$ are Riesz bases and if $(V_i , v_i )$ is a fusion Riesz basis, then $D_V$ and $D_{\varphi ^{(i)}}$ ($i\in I$) are bounded and bijective. By Lemma \ref{pseudolemma}, $\bigoplus _{i\in I} D_{\varphi ^{(i)}}$ is bounded and bijective. Therefore the composition $D_{v\varphi} = D_V \bigoplus _{i\in I} D_{\varphi ^{(i)}}$ is bounded and bijective, i.e. $v\varphi$ is a Riesz basis.

The claim for the Riesz bounds follows from Theorem \ref{start}.
\end{proof}

Consider a fusion frame system $(V_i , v_i , \varphi ^{(i)} )_{i \in I}$ with corresponding global frame $v\varphi$. In \cite{cakuli08}, the authors posed the question, when centralized reconstruction equals distributed reconstruction, i.e. when the dual frame $(S_V ^{-1} S_{\varphi ^{(i)}} ^{-1} v_i \varphi_{ij} )_{i\in I, j\in J_i}$ of $v\varphi$ coincides with the canonical dual frame $(S_{v\varphi}^{-1} v_i \varphi _{ij})_{i\in I, j\in J_i}$. In \cite{cakuli08} the authors showed that this holds, if $(V_i)_{i\in I}$ is an orthonormal fusion basis. Observe that this also is true if we instead assume that $v\varphi$ is a Riesz basis, since, in this case, the canonical dual is the unique dual frame for $v\varphi$ \cite{ole1} and the question becomes trivial. In the next result, we only assume $(V_i , v_i )_{i \in I}$ to be a fusion Riesz basis, which is weaker than both of the previously mentioned assumptions (compare to Theorem \ref{great2}) and show that also in this case centralized reconstruction equals distributed reconstruction. 

\newtheorem{fusframesysthm}[lemmaisometry]{Theorem}

\begin{fusframesysthm}\label{fusframesysthm}
Let $(V_i , v_i , \varphi ^{(i)})_{i \in I}$ be a fusion frame system and $v\varphi$ be its corresponding global frame. If $V=(V_i , v_i )_{i\in I}$ is a fusion Riesz basis, then $(S_V ^{-1} S^{-1}_{\varphi ^{(i)}} v_i \varphi _{ij} )_{i\in I,j\in J_i}$ is the canonical dual frame of $v\varphi$, i.e.
$$(S_{v\varphi}^{-1} v_i \varphi _{ij} )_{i\in I, j\in J_i} 
= (S_V ^{-1} S^{-1}_{\varphi ^{(i)}} v_i \varphi _{ij} )_{i\in I, j\in J_i} .$$
\end{fusframesysthm}

\begin{proof}
By Lemma \ref{inverseoperators}, we have $D_V^{-1} v_i \varphi_{ij} = (..., 0, 0, \varphi_{ij}, 0, 0, ...)$ ($\varphi_{ij}$ in the $i$-th entry). By applying Proposition \ref{dualglobal}, this yields 
\begin{flalign}
S_{v\varphi}^{-1} v_i \varphi _{ij} &= C_V^{-1} \big( \bigoplus _{k\in I} S^{-1}_{\varphi ^{(k)}} \big) D^{-1}_V v_i \varphi _{ij} \notag \\
&= C_V^{-1} \big( \bigoplus _{k\in I} S^{-1}_{\varphi ^{(k)}} \big) (..., 0, 0, \varphi_{ij}, 0, 0, ...)  \notag \\ 
&= C_V^{-1} (..., 0, 0, S^{-1}_{\varphi ^{(i)}} \varphi_{ij}, 0, 0, ...)  \notag \\ 
&= C_V^{-1} D_V^{-1} D_V (..., 0, 0, S^{-1}_{\varphi ^{(i)}} \varphi_{ij}, 0, 0, ...)  \notag \\ 
&= S_V ^{-1} S^{-1}_{\varphi ^{(i)}} v_i \varphi _{ij} . \notag 
\end{flalign}
Since this is true for all $i\in I$ and $j\in J_i$, the proof is finished.
\end{proof}

%
%
%
%
%
%
%
We conclude with the following collecting results on inheritable structures in fusion frame systems. Recall that we denote the common frame bounds of the local frames in a fusion frame system by $A$ and $B$.

\newtheorem{structures}[lemmaisometry]{Corollary}

\begin{structures}\label{structures} 
Let $(V_i , v_i , \varphi ^{(i)} )_{i \in I}$ be a fusion frame system for $\mathcal{H}$ and $v\varphi$ be the corresponding global frame.
   \begin{itemize}
       \item[(1)] \cite{caskut04} If $\varphi ^{(i)}$ is Parseval for every $i\in I$, then $S_{v\varphi} = S_V$.
       \item[(2)] If $\varphi ^{(i)}$ is $A$-tight for every $i\in I$, then $S_{v\varphi} = A\cdot S_V$.
       \item[(3)] If $\varphi ^{(i)}$ is $A_i$-tight for every $i\in I$, then $S_{v\varphi} = S_W$, where $W = (V_i, \sqrt{A_i}v_i)_{i\in I}$.
   \end{itemize}
\end{structures}

\begin{proof}
(2) By (\ref{frame}), $S_{v\varphi} = D_V \big( \bigoplus_{i\in I} A \mathcal{I}_{V_i}\big) C_V = D_V A\mathcal{I}_{( \sum_{i\in I} \oplus V_i )_{\ell^2}} C_V = AD_V C_V = AS_V$. (3) follows similarly by observing that for any $f\in \mathcal{H}$ we have $S_{v\varphi} f = D_V \big( \bigoplus_{i\in I} A_i \mathcal{I}_{V_i}\big) C_V f = \sum_{i\in I} A_i v_i^2 \pi_{V_i} f = S_W f$.
\end{proof}

\newtheorem{structures2}[lemmaisometry]{Corollary}

As an immediate consequence we obtain the following:

\begin{structures2}\label{structures2} Let $(V_i , v_i , \varphi ^{(i)} )_{i \in I}$ be a fusion frame system for $\mathcal{H}$ and $v\varphi$ be the corresponding global frame.
   \begin{itemize}
   \item[(1)] \cite{cakuli08} If $\varphi ^{(i)}$ is Parseval for every $i\in I$, then $V$ is Parseval if and only if $v\varphi$ is Parseval.
   \item[(2)] \cite{cakuli08} If $\varphi ^{(i)}$ is Parseval for every $i\in I$, then $V$ is $A$-tight if and only if $v\varphi$ is $A$-tight.
   \item[(3)] If $\varphi ^{(i)}$ is $A$-tight for every $i\in I$, then $V$ is tight if and only if $v\varphi$ is tight.
   \item[(4)] In case $\mathcal{H}$ is a RKHS: If $\varphi ^{(i)}$ is Parseval for every $i\in I$, then $V$ is painless if and only if $v\varphi$ is painless.
   \item[(5)] In case $\mathcal{H}$ is a RKHS: If $\varphi ^{(i)}$ is $A$-tight for every $i\in I$, then $V$ is painless if and only if $v\varphi$ is painless.
   \end{itemize}
\end{structures2}

\newtheorem{structuresB}[lemmaisometry]{Corollary}

%
We can also show a converse of Corollary \ref{structures} with restrictions to the fusion frame $V$.

\begin{structuresB}\label{structuresB}
Let $(V_i , v_i , \varphi ^{(i)} )_{i \in I}$ be a fusion frame system and $v\varphi$ be the corresponding global frame.
\begin{itemize}
\item[(1)] If $S_{v\varphi} = S_V$, then $\bigoplus_{i\in I} (\mathcal{I}_{V_i} - S_{\varphi^{(i)}}) = 0$ on $\mathcal{R}(C_V)$. In particular, if $V$ is a fusion Riesz basis and $S_{v\varphi} = S_V$, then $\varphi ^{(i)}$ is Parseval for every $i\in I$.
\item[(2)] If $S_{v\varphi} = A\cdot S_V$, then $\bigoplus_{i\in I} (A\cdot \mathcal{I}_{V_i} - S_{\varphi^{(i)}}) = 0$ on $\mathcal{R}(C_V)$. In particular, if $V$ is a fusion Riesz basis $S_{v\varphi} = A\cdot S_V$, then $\varphi ^{(i)}$ is $A$-tight for every $i\in I$.
\item[(3)] If $S_{v\varphi} = S_W$, where $W = (V_i, \sqrt{A_i}v_i)_{i\in I}$, then $\bigoplus_{i\in I} (A_i \cdot \mathcal{I}_{V_i} - S_{\varphi^{(i)}}) = 0$ on $\mathcal{R}(C_V)$. In particular, if $V$ is a fusion Riesz basis and $S_{v\varphi} = S_W$, then $\varphi ^{(i)}$ is $A_i$-tight for every $i\in I$.
\end{itemize}
\end{structuresB}

\begin{proof}
(1) By (\ref{frame}), we have $S_V - S_{v\varphi} = D_V\bigoplus_{i\in I} (\mathcal{I}_{V_i} - S_{\varphi^{(i)}}) C_V = 0$. This implies that
\begin{flalign}
0 &= \Big\langle D_V \bigoplus_{i\in I} (\mathcal{I}_{V_i} - S_{\varphi^{(i)}}) C_V f, f \Big\rangle_{\mathcal{H}} \notag \\
&= \Big\langle \bigoplus_{i\in I} (\mathcal{I}_{V_i} - S_{\varphi^{(i)}}) C_V f, C_V f \Big\rangle_{(\sum_{i\in I}\oplus V_i)_{\ell^2}} \quad (\forall f\in \mathcal{H}). \notag
\end{flalign}
Since $\bigoplus_{i\in I} (\mathcal{I}_{V_i} - S_{\varphi^{(i)}})$ is self-adjoint by Lemma \ref{Op2} (b), this implies $\bigoplus_{i\in I} (\mathcal{I}_{V_i} - S_{\varphi^{(i)}}) = 0$ on $\mathcal{R}(C_V)$. If $V$ additionally is a fusion Riesz basis, then $\mathcal{R}(C_V) = (\sum_{i\in I}\oplus V_i)_{\ell^2}$, which implies $S_{\varphi^{(i)}} = \mathcal{I}_{V_i}$ for all $i\in I$. (2) and (3) can be shown analogously.
\end{proof}
		
\newtheorem{structures3}[lemmaisometry]{Corollary}

\section*{Acknowledgments}

The authors thank Mitra Shamsabadi, Karin Schnass and Hans Feichtinger for related discussions. The authors also thank the anonymous reviewers for their helpful comments and suggestions. 

This work is supported by the project P 34624 \emph{"Localized, Fusion and Tensors of Frames"} (LoFT) of the Austrian Science Fund (FWF).

\section*{Conflict of interest}

On behalf of all authors, the corresponding author states that there is no conflict of interest.

\bibstyle{abbrv}
\bibliography{biblioall}
\end{document}